\documentclass[11pt]{article}
\usepackage{amsthm}
\usepackage{amsfonts,amsmath,amstext,amsbsy,euscript,amssymb, graphics}
\usepackage{rotating}
\usepackage{hyperref}
\usepackage{showlabels}

\newtheorem{theorem}{Theorem}
\newtheorem{lemma}[theorem]{Lemma}
\newtheorem{cor}[theorem]{Corollary}
\newtheorem{prop}[theorem]{Proposition}
\newtheorem{conj}[theorem]{Conjecture}
\theoremstyle{definition}
\newtheorem{nota}{Notation}
\newtheorem{remark}{Remark}

\newtheorem{example}{Example}
\def\x{\ensuremath{\boldsymbol{x}}}
\def\y{\boldsymbol{y}}

\def\s{\boldsymbol{s}}
\def\r{\boldsymbol{r}}
\def\f{\boldsymbol{f}}

\def\t{\boldsymbol{t}}
\def\m{\ensuremath{\boldsymbol{m}}}
\def\0{\mathbf{0}}
\def\N{\mathbb N}
\def\M{{\cal M}}
\def\R{{\cal R}}
\def\F{{\cal F}}
\def\rf{\right\rfloor}
\def\lf{\left\lfloor}
\newcommand{\rat}[2]{\lf\frac{#1}{#2}\rf}

\newtheorem{definition}{Definition}

\begin{document}
\title{Playability and arbitrarily large rat games}
\author{Aviezri S. Fraenkel,\\ Dept. of Computer Science and Applied Mathematics,\\ Weizmann Institute of Science,\\ Rehovot 76100, Israel; fraenkel@wisdom.weizmann.ac.il\\\\
Urban Larsson\\
Dept. of Industrial Engineering and Management,\\ Technion--Israel Institute of Technology,\\ Haifa 3200003, Israel; urban031@gmail.com}
\maketitle
\begin{abstract}
In 1973 Fraenkel discovered interesting sequences which split the positive integers. These sequences became famous, because of a related unsolved conjecture. Here we construct combinatorial games, with `playable' rulesets, with these sequences constituting the winning positions for the second player. Keywords: Combinatorial game, Fraenkel's conjecture, Impartial game, Normal play, Playability, Rational modulus, Splitting sequences.
\end{abstract}
\section{Introduction}
We study 2-player combinatorial heap games, which generate sequences of non-negative integer vectors in form of `winning strategies' a.k.a.  `P-positions'. The games are acyclic impartial combinatorial games, with alternating play. They have perfect information, and it is well known that one can partition the game positions into previous player winning positions (P-positions) and next player winning positions (N-positions). Here we use the normal play convention: a player unable to move, loses. %Two games are \emph{P-equivalent} if they have the same set of P-positions.

A basic problem in combinatorial game theory is to find an \emph{efficient winning strategy} for a game, and this boils down to two problems. Decide if a position is a P-position in polynomial time (in succinct input size), and if not a P-position, then find a winning move in polynomial time.

On the other hand, famous sequences can sometimes be associated with games. But this is a less obvious statement \cite{LarThesis}, \cite{Fra2004}: when does a sequence of vectors of non-negative integers have interesting game rules, such that a winning strategy is given by the sequence?

The early connection between Wythoff nim and complementary sequences of modulus the golden ratio and its square respectively, recently lead to research in finding game rules for any complemenatary pair of homogenous Beatty sequences of irrational modulus \cite{Bea1926} (generalizing Wythoff's sequences) \cite{LHF2011}. The solution is appealing, for example since it introduced a new operator to combinatorial game theory, but unfortunately it is not known whether the rules of game can be understood in polynomial time (in succinct input size). We arrive at a motivation for this paper: rulesets for combinatorial games should be suitable also for players without a degree in mathematics (even many games with great theoretical value have this property, e.g. \cite{BCG1982}), and to this purpose, in Section~\ref{sec:games}, we will define a concept of \emph{playability} for (multi-pile) heap games.  More overview: 

In Section~\ref{sec:rules}, we define \emph{succinct game rules}. The class of sequences of interest are the \emph{rat-vectors} (``rat" for rational modulus), and they are  defined in Section~\ref{sec:sequences}.

In Section~\ref{sec:games}, we define the \emph{grandiose rat games} as vector-subtraction games \cite{Gol1966}, %, and in Section~\ref{sec:existence}, we show that the P-positions for these games are the rat-vectors. 
and in Section~\ref{sec:playable}, we show that the games in Section~\ref{sec:rules} and Section~\ref{sec:games} are the same. 

In Section~\ref{sec:binary}, we give a matrix representation of the rat-vectors, and then, in Section~\ref{sec:ternary}, we build matrices for so-called \emph{shortcut-vectors} which connect the pairs of rat-vectors via subtraction. 

In Section~\ref{sec:approxnim}, we show that, in a specific sense, the games are close to the game of nim, and in Section~\ref{sec:rightshift}, we study a `right-shift' property of the rat vectors. 

At last, in the Appendix, we supply relevant figures, including data and conjectures for future work.

\section{Succinct rules for rat games}\label{sec:rules}
The rules of our \emph{succinct games} are as follows. Let $d\ge 2$ be an integer. We play on $d$-tuples (vectors) of non-negative integers $\x = (x_1,\ldots , x_d)$. The move options are \emph{vector subtractions}, and any vector subtraction $\x - \s = (x_1-s_1,\ldots , x_d-s_d) \ge \0$ is allowed, with $\s = (s_1,\ldots , s_d)$, except if it satisfies either of the following two properties, a or b :
\begin{itemize}
\item[a(i)] $s_d - 2^{d-1}$ is a multiple of $2^{d}-1$, and 
\item[a(ii)] for all $i\in \{2,\ldots , d\}$, $s_i-1\le 2s_{i-1}\le s_{i}$,
\item[] or
\item[b(i)] $s_d$ is a multiple of $2^{d}-1$, and 
\item[b(ii)] for all $i\in \{2,\ldots , d\}$, $s_i-1\le 2s_{i-1}\le s_{i}+1$.
\end{itemize}
Say, $d=3$ and the starting position is $(1,3,7)$. There is no move to $\0$, since condition b is satisfied by $s_3 = 7=2^3-1$ and since $2s_{i-1} = s_{i} + 1$, for $i = 2, 3$; $2\times 1+1 = 3$ and $2\times 3 +1 =7$. However, there is a move to $(1,2,4)$, since $(0, 1, 3)$ satisfies neither a nor b. But $(1,2,4)$ is the smallest (using lexicographic order) position of the forms a or b, which implies that the next move will be a losing move. Hence position $(1,3,7)$ is an N-position, a winning position for the current player. 

\section{Fraenkel's popular rat sequences}\label{sec:sequences}
 Let $\N=\{1,2,\ldots\}$ denote the positive integers, and let $\N_0 =\N \cup \{0\}$.
The \emph{rat sequences} are of the form $(\rat{3n}{2}, 3n-1)$, $(\rat{7n}{4}, \rat{7n}{2}-1, 7n-3)$, $(\rat{15n}{8}, \rat{15n}{4}-1, \rat{15n}{2}-3, 15n-7)$, and so on, for $n\in \N$. Thus, for each dimension $d\ge 2$, we code the vectors by $\r(n) = (r_{1}(n), \ldots , r_{d}(n))$, $n\in \N$, where
\begin{align}\label{eq:rat}
r_i(n) = \rat{(2^d-1)n}{2^{d-i}}-2^{i-1} + 1,
\end{align}
$i\in \{1,\ldots , d\}$. This representation will be referred to as the \emph{standard form}, and, for each dimension $d\ge 2$, we can think of it as an infinite row-matrix on $d$ columns, with rows splitting the positive integers (see Theorem~\ref{thm:split} below).  

Note that each column is \emph{arithmetic periodic} with saltus $2^d-1$ and period $2^{d-i}$. This property motivates us to introduce matrix representations (Section~\ref{sec:binary}) for the rat sequences, and thus study games and solutions in their finite representations.

\subsection{Rat history}
The results in this paper do not depend on the material of this subsection, which is included to show the historical and mathematical value of the rat-vectors. This history provides some motivation for this paper. 

So-called Beatty sequences \cite{Bea1926, BOHA27} are normally associated with irrational {\it moduli\/} $\alpha$, $\beta$. Recent studies deal with rational moduli $\alpha$, $\beta$. Clearly if $a/b\ne g/h$ are rational, then the sequences $\{\lfloor na/b\rfloor\}$ and $\{\lfloor ng/h\rfloor\}$ cannot be complementary, since $kbg\times a/b= kha\times g/h=kag$ for all $k\ge 1$. Also the former sequence is missing the integers $ka-1$ and the latter $kg-1$, so both are missing the integers $kag-1$ for all $k\ge 1$. However, complementarity can be maintained for the {\it nonhomogeneous\/} case: In \cite{Fra1969}, \cite{Obr2003}, necessary and sufficient conditions on $\alpha$, $\gamma$, $\beta$, $\delta$ are given so that the sequences $\{\lfloor n\alpha+\gamma\rfloor\}$ and $\{\lfloor n\beta+\delta\rfloor\}$ are complementary -- for both irrational moduli and rational moduli. We are not aware of any previous work in this direction, except that in Bang \cite{Ban1957} necessary and sufficient conditions are given for $\{\lfloor n\alpha\rfloor\}\supseteq\{\lfloor n\beta\rfloor\}$ to hold, both for the case $\alpha$, $\beta$ irrational and the case $\alpha$, $\beta$ rational. Results of this sort also appear in Niven \cite{Niv1963}, for the homogeneous case only. In Skolem \cite{Sko11957} and Skolem \cite{Sko21957} the homogeneous and nonhomogeneous cases are studied, but only for $\alpha$ and $\beta$ irrational. Fraenkel formulated  the following conjecture:\footnote{Erd\"os and Graham mention the conjecture in \cite{ErGr1980} (p.\ 19), as well as Graham et.~al.\ in \cite{Gra1978}. It is also a research problem in `Concrete Mathematics' by Graham, Knuth, Patashnik \cite{GrKnPa1989} (ch.~3), and is mentioned by Tijdeman \cite{Tij12000}, \cite{Tij2000}. In \cite{Fra1973} a weaker conjecture, implied by the full conjecture, is formulated and proved for special cases: If $d\ge 3$, then there are always two distinct moduli with integral ratio. Simpson proved it when one of the moduli (and hence the only one) is $\le 2$ \cite{Sim1991}.} 

\begin{conj}
If the vectors $(\lfloor n\alpha_i+\gamma_i\rfloor)_{i=1}^d$, $n\in\N$ split the positive integers with $d\ge 3$ and $\alpha_1<\alpha_2<\ldots <\alpha_d$, then
\begin{eqnarray}\label{conj}
\alpha_i=(2^d-1)/2^{d-i},\ i=1, \ldots, d.
\end{eqnarray}
\end{conj}

Fraenkel \cite{Fra1973} proved that this system of vectors partitions (splits) the positive integers with explicit values for $\gamma_i$ as in (\ref{eq:rat}). 

\begin{theorem}[Fraenkel 1973]\label{thm:split}
For any dimension $d\ge 2$, $$\N = \{r_j(n)\mid n\in \N, j\in \{1, \ldots , d\}\}$$ and $r_i(n) = r_{j}(m)$ implies $(i, n) = (j, m)$.
\end{theorem}

It is well-known that if all the $\alpha_i$ are integers with $d\ge 2$ and $\alpha_1\le \alpha_2\le\ldots \le\alpha_d$, then $\alpha_{d-1}=\alpha_d$. A generating function proof using a primitive root of unity was given by Mirsky, Newman, Davenport and Rado -- see Erd\"os \cite{Erd1952}. A first elementary proof was given independently in \cite{BeFeFr1986} and by Simpson \cite{Sim1986}. Graham \cite{Gra1973} showed that if one of the $d$ moduli is irrational then all are irrational, and if $d\ge 3$, then two moduli are equal. Thus distinct integer moduli or distinct irrational moduli cannot exist for $d\ge 2$ or $d\ge 3$ respectively in a splitting system.

The conjecture was proved for $d=3$ by Morikawa \cite{Mor1982}, $d=4$ by Altman et.\ al \cite{AlGaHo2000}, for all $3\le d\le 6$ by Tijdeman \cite{Tij2000} and for $d=7$ by Bar\'at and Varj\'u \cite{BarVar2003} and was generalized by Graham and O'Bryant \cite{GraObr2005}. Other partial results were given by Morikawa \cite{Mor1985}, Simpson \cite{Sim2004}. Many others have contributed partial results -- see Tijdeman \cite{Tij12000} for a detailed history. The conjecture has some applications in job scheduling and related industrial engineering areas, in particular: `Just-In-Time' systems, see e.g., Altman et.\ al \cite{AlGaHo2000}, Brauner and Jost \cite{BrJo2008}, Brauner and Crama \cite{BrCr2004}. However, the conjecture itself has not been settled. So this is a problem that has been solved for the integers, has been solved for the irrationals, and is wide open for the rationals!

The conjecture, with accomapnying Theorem~\ref{thm:split}, induced the ``rat game'' and its associates the ``mouse game'' \cite{Fra2015} (rat -- rational), played on 3 and 2 piles of tokens respectively, whose $P$-positions are the cases $d = 2,3$ of definition (\ref{eq:rat}) respectively, together with $\0$. However, arguably, those rules are only intended for players with a degree in mathematics, and they cannot be described as a \emph{vector subtraction game}---a natural notion, including many classical games introduced by Golomb \cite{Gol1966}. Apart from the ending condition, the moves of a vector subtraction game are independent of the size of the heaps. In response to the heap size dependency of the mouse game, a vector subtraction game on two heaps, dubbed \emph{the mouse trap} \cite{Lartrap} was developed, using the so-called $\star$-operator \cite{LHF2011}. Indeed, the inaccessibility of those rules, and the difficulty of generalization, further motivats our approach.

\section{Grandiose rat games and playability}\label{sec:games}

Let $d\in \N$. We let $\M\subset {\N_0\!}^d$ describe (a set of moves of) a \emph{vector subtraction game} of the form: from each position $\x\in {\N_0\!}^d$, there is a move to position $\y\in {\N_0\!}^d$ if and only if $\x - \y\in \M$. We say ``\m\ is a move" if $\m\in \M$.

For a fixed dimension $d>1$, let $\R = \{\r(n)\mid n\in \N\}\cup \{\0\}$ be the candidate set of P-positions, and let $\R-\R = \{\r-\s\mid \r,\s\in \R\}$ be the set of shortcut-vectors, or \emph{shortcuts}. Let $(\R-\R) \setminus \R$ be the set of \emph{proper shortcuts}.

From now onwards, we let $d\ge 2$, and we identify $\M = {\N_0\!}^d\setminus (\R-\R)$ with the vector subtraction game, where any move is available except if it shortcuts two vectors in $\R$. Note that $\0\not \in \M$, so all games terminate in a finite number of moves. The class of all such games is the class of \emph{grandiose rat games} (or grandiose games).

In Section~\ref{sec:existence}, Theorem~\ref{thm:existence}, we show that, for any number of heaps $d \ge 2$, $P(\M)=\R$, and in Section~\ref{sec:playable}, we prove that the succinct games and the grandiose games are the same.

A nice property for a ruleset defined on any number of heaps is that the description on how to move does not increase too fast when the number of heaps grows. Since the modulus is rational, our winning strategies and games are \emph{arithmetic periodic}, and we will use matrix representations to hightlight this fact. For  each dimension $d$, and all $n\in \N$, we have that
\begin{itemize}
\item the number of rat-vectors with coordinates less than or equal $2^{d-1}(2^d-1)n$ is $2^{d-1}n$,
\item the number of proper shortcuts with coordinates less than or equal $2^{d-1}(2^d-1)n$ is $(3^{d-1})n < 2\binom{2^{d-1}}{2}n$ (many get canceled, and we show this in Section~\ref{sec:ternary}).
\end{itemize}

So, by arithmetic periodicity, if we give the job to the previous player to refute any suggested vector not in $\M$, then, by exhaustive search the number $3^{d-1}+2^{d-1}$ is an upper bound. This number is constant in the heap sizes, but still exponential in the number of heaps. By this alone, the grandiose games do not appear playable for a large number of heaps.

Here, we prove that there is a much faster way to refute a move, namely, by a linear procedure, in the number of heaps. In fact, in Section~\ref{sec:playable} we prove that the succinct rules, with its simple procedure makes also the grandiose games playable by anyone with an elementary knowledge in arithmetics, for any number of heaps.

Note that, the quintessence of impartial combinatorial games, Nim, has a linear time procedure in the number of heaps to decide whether a given vector is a move\footnote{Of course, for Nim we must disallow any vector subtraction if more than one coordinate is positive (or if a positive coordinate is larger than the corresponding heap size).} (if heap sizes are bounded) and so does of course Moore's Nim, and some other variations of Nim such as Fraenkel's multi-pile generalization of Wythoff Nim \cite{Fra1996} (with accompanying conjecture). Not all game rules are defined over an arbitrarily finite number of heaps, but for those classes, where it is applicable, we suggest the following terminology.

\begin{definition}
The rules of a heap game are \emph{playable} if (for bounded heap sizes) they satisfy a linear procedure in the number of heaps, and (for bounded number of heaps) a log-linear procedure in the heap sizes\footnote{This is trivially required since, for example, we must check that move coordinates are not larger than heap sizes. Apart from this requirement, we will see that, in our case, the complexity of the rules is \emph{constant in the size of the heaps.}}. 
\end{definition}

\section{Playability for grandiose games and succinct games}\label{sec:playable}
We prove that the grandios rat games are playable, for any number of heaps, by linking them to the succinct games.

When we shift a divisor 2 inside the floor function, then the deviation is small.

\begin{lemma}\label{lem:shift}
For any number $y$,
\begin{align}\label{eq:1}
0\le \left\lfloor y\right\rfloor / 2 - \left\lfloor\frac{y}{2}\right\rfloor\le \frac{1}{2}.
\end{align}
\end{lemma}

\begin{proof}
Since $y\ge 2\lfloor \frac{y}{2}\rfloor$, for any $y$, then $\lfloor y\rfloor -2\lfloor \frac{y}{2}\rfloor > (y-1)-y=-1$. Since the expression is an integer, by the strict inequality, the lower bound holds. The upper inequality follows by decompositioning into fractional parts. Put $\lfloor y\rfloor -2\lfloor \frac{y}{2}\rfloor = y-(y - \lfloor y\rfloor) - 2(\frac{y}{2}-(\frac{y}{2}-\lfloor \frac{y}{2}\rfloor))=(y - \lfloor y\rfloor) - 2(\frac{y}{2}+\lfloor \frac{y}{2}\rfloor))\le 1$.
\end{proof}

We say that a $d$-tuple $\boldsymbol x = (x_1,\ldots ,x_d)$ has a \emph{ternary recurrence}\footnote{See also Theorem~\ref{thm:shortcut} in Section~\ref{sec:ternary}.} if it satisfies $x_d\equiv 0 \pmod {2^d-1}$, and,  for all $i\in \{2, \ldots , d\}$,  $x_{i - 1}\in \left\{\left\lfloor \frac{x_i}{2}\right\rfloor,\left\lceil \frac{x_i}{2}\right\rceil\right\}.$ Note that this corresponds to the succinct game rules. The ``if'' part of the the following result depends on the subsequent sections (in particular Section~\ref{sec:ternary}); however, the ``only if'' direction is independent of later results and, as we will see, it implies the connection between the succinct and grandiose games.

\begin{theorem}\label{thm:playable}
A vector $\x$ is a proper shortcut if and only if it has a ternary recurrence.
\end{theorem}

\begin{proof}
%Write $$r_i(n) = \left\lfloor\frac{(2^d - 1)n}{2^{d - i}}\right\rfloor.$$
By definition, the vector $\boldsymbol x$ is a shortcut, if, for all $i$, for some $k>0$, $x_i = r_i(n+k)-r_i(n)$, with $r_i$ defined as in (\ref{eq:rat}). This gives $x_d = k(2^d-1)$, so the congruence part holds. 

Next, we prove that $0 \le x_{i-1}- \left\lfloor \frac{x_i}{2}\right\rfloor \le 1,$ for all $i$, if $\boldsymbol x$ is a shortcut. Let $\varphi = x_{i-1}- \left\lfloor \frac{x_i}{2}\right\rfloor $.

If $\x$ is a shortcut, then $$\varphi = r_{i-1}(n+k) - r_{i-1}(n) - \left\lfloor \frac{ r_{i}(n+k) - r_{i}(n)}{2}\right\rfloor . $$
How much does the second term differ from the first? Note that, if we \emph{shift} the divisor in the second term inside the inner floor functions, then we get $r_{i-1}(n+k) - r_{i-1}(n) - \left\lfloor r_{i-1}(n+k) - r_{i-1}(n) \right\rfloor = 0, $
because $r$ is an integer valued function. Hence, by applying Lemma~\ref{lem:shift} in shifting back the divisor 2 outside the inner most nested floor functions, we get  $-1\le \varphi \le 1$.

Since the expression is an integer, it suffices to exclude the case $\varphi = -1$, so, let us assume, for a contradiction, that
$$ 1+r_{i-1}(n+k) - r_{i-1}(n) = \left\lfloor \frac{ r_{i}(n+k) - r_{i}(n)}{2}\right\rfloor . $$

Again, by using Lemma~\ref{lem:shift}, the expression inside the right hand side floor function has increased at most a half, by moving the divisor 2 outside the inner floor functions. Since the expression was an integer before this operation, the possible increase of a half will be canceled by the outer floor function. Thus the left hand side is too large.

We have proved that, if $\x$ is not of the form in the second part of the theorem, then $\x$ is a move.

For the other direction, suppose that  $$x_d\equiv 0 \pmod {2^d-1}$$ and, for all $i\in \{2, \ldots , d\}$,  $$x_{i - 1}\in \left\{\left\lfloor \frac{x_i}{2}\right\rfloor,\left\lceil \frac{x_i}{2}\right\rceil\right\}.$$

We have to demonstrate that there exist $n$ and $k$ such that $\x = \r(n+k)-\r(n)$. We use the shortcut-matrix defined in Section~\ref{sec:ternary}. Since each ternary vector defines uniquely each row, and starting with $x_d$, the existence is clear.
\end{proof}

In the Appendix, using figures, data, and conjectures, we show that the stucture of the shortcut matrices, for increasing $d$, satisfy interesting Cantor-like line fractals, with apparent complex behavior.  Disregarded this apparent complexity, we prove that the succinct games are the same as the grandiose games. 

%%%%This part is intended for integers%%%%%%%
%In an accompaying arXiv manuscript \cite{RatData}, using figures, data, and conjectures, we show that the stucture of the shortcut matrices, for increasing $d$, satisfy interesting Cantor-like line fractals, with somewhat complex behavior; independently of this complexity, we prove that the succinct games are the same as the grandiose games. 

\subsection{Linking the games}\label{sec:existence}
If you add a pair of rat vectors, then the result is never a proper shortcut.
\begin{lemma}\label{lem:1}
Let $\r\in \R$. Then $\r+\r'\ne \s-\s'$, for any $\r',\s,\s'\in \R\setminus \{\0\}$.
\end{lemma}
\begin{proof}
If $\r = \0$, then $r_d + r'_d \equiv 2^{d-1} \pmod{2^{d} - 1}$, and otherwise, $r_d + r'_d \equiv 2^{d-1} + 2^{d-1}\equiv1\pmod{2^{d} - 1}$, but $s_d - s'_d \equiv 2^{d-1} - 2^{d-1}\equiv 0\pmod{2^{d} - 1}$.
\end{proof}
%The proof of existence of rules for the rat sequences follows from our method of deriving playability, so in the proof of existence, we use a definition (\emph{ternary recurrence}) and a tool from Section~\ref{sec:playable}. 
We have the following corollary of Theorem~\ref{thm:playable}.
\begin{cor}\label{lem:2}
Let $\x \in {\N_0\!}^d\setminus \R$. 
\begin{itemize}
\item[(i)] If $\r\in \R$ then $\x - \r$ is not a shortcut. 
\item[(ii)] If $\r\in \R\setminus \{\0\}$ then $\x - \r$ is not a proper shortcut.
\end{itemize}
\end{cor}
\begin{proof}
Item (ii) follows by Theorem~\ref{thm:playable}, since each shortcut has ternary recurrence, but $\x-\r$ does not have ternary recurrence (since $\x\not \in \R$ but $\r\in \R$). Then item (i) follows by includding the case $\r=\0$.
\end{proof}
\begin{theorem}\label{thm:existence}
For any number of heaps $d > 1$, $P(\M) = \R$.
\end{theorem}
\begin{proof}
The property: no candidate P-position has a move to another candidate P-position is immediately satisfied by the definitions of $\R$ and $\M$.

Hence, it suffices to prove that each candidate N-position $\x\in {\N_0\!}^d\setminus \R$ has a move to a candidate P-position $\r\in \R$. Thus, we have to find an $\m\in \M$ such that $\x - \m = \r$, for some $\r\in \R$. Note that, if we find an $\r\in\R$ such that $\x-\r=\r' \in \R$, then $\x = \r+\r'\in \M$, by Lemma~\ref{lem:1} and since $\x\ne \r\in \R$, so there is a move to $\0$. Assume therefore that, for all $\r\in \R$, $\x -\r\not\in \R$. By Corollary~\ref{lem:2} (ii), $\x-\r$ is not a proper shortcut, so altogether $\x-\r\in \M$.
\end{proof}
One can also prove Theorem~\ref{thm:existence} directly from the definition of the rat sequences (\ref{eq:rat}), i.e. without using Theorem~\ref{thm:playable}, and we encourage the reader to try it out.
\begin{cor}
For a fixed $d\ge 2$, the succinct game (from Section~\ref{sec:rules}) is the grandiose game $\M$. That is, the succint game rules suffices to play the grandiose game. For both games there is a constant time (in the heap sizes) and linear time (in the number of heaps) procedure to decide whether a given $d$-tuple is a move. 
\end{cor}
\begin{proof}
This follows from Theorem~\ref{thm:playable} and Theorem~\ref{thm:existence}.
\end{proof}

Let us give another play example, here with $d=4$. 
\begin{example}
Let $\x = (4, 7, 15, 29)$. Then $x_4 \not\equiv 0 \pmod{2^4-1}$. So $\x$ is a move. Let $\x = (4, 7, 15, 30)$. Then $x_4 \equiv 0 \pmod{2^4-1}$. In addition $30/2=15, \lfloor 15/2\rfloor =7$ and $\lceil 7/2\rceil = 4$, so $\x\in \R-\R$ is a shortcut.

How do you move from $(3,6,12,23)+(4, 7, 15, 30)=(7,13,27,53)$? The first vector is a P-position, but the second is a shortcut. Is there any other attainable P-position?
We must find a move of the form $(7-\rat{15n}{8}, 14-\rat{15n}{4}, 30-\rat{15n}{2}, 60-15n)$. The forth coordinate is correct, so we proceed by dividing by 2 and applying the floor function, and thus verify the third coordinate (for $n = 1,2,3$). For the second coordinate: is there an $n$ such that $$ 14-\rat{15n}{4}\in \left\{\rat{30-\rat{15n}{2}}{2},\rat{30-\rat{15n}{2}}{2}+1\right\}?$$
It turns out that $n=1$ also gives a shortcut, but for $n=3$, the move $(2,3,8,15)$ takes you to the P-position $(5,10,19,38)$. This type of positions makes the game not only playable from a trivial point of view, but hopefully also enjoyable, because good (but non-optimal) players will strive to be close to shortcuts, and there could be small  mistakes which flips the game over to the opponent.
\end{example}

%%%%%%%%%%%%%%%%%%%%%%%%%%%%%%%MATRIX%%%%%%%%%%%%%%%%%%%%%%%%%%%%%%%%%%%%%%%%%%%%%%%%
\section{The anatomy of rats: matrix representations}\label{sec:binary}

We have a general observation on the floor function.
\begin{lemma}\label{lem:xy}
For any $x,y\in \N$, $\rat{x}{y}-\rat{x-1}{y}= 1$, if $x \equiv 0\pmod{y}$, and otherwise $\rat{x}{y}-\rat{x-1}{y} = 0$.
\end{lemma}
\begin{proof}
Write $x = \alpha y +\beta$, with $0\le \beta < y$. This gives $\rat{x}{y}-\rat{x-1}{y} = \rat{\alpha y + \beta}{y}-\rat{\alpha y + \beta-1}{y} = \rat{\beta}{y}-\rat{\beta-1}{y}=1$ if and only if $\beta = 0$.
\end{proof}
\begin{nota}
For each column $j\in \{1,\ldots , d\}$, the \emph{gap} between the rows $n\ge 2$ and $n-1$ is $\Delta_j(n) := r_j(n)-r_j(n-1)$.
\end{nota}
\begin{lemma}\label{lem:Delta}
For all $n, j$, $\Delta_j(n) = 2^j$, unless $n\equiv 0\pmod {2^{d-j}}$, in which case $\Delta_j(n) = 2^j-1$.
\end{lemma}
\begin{proof}
For all $n,j$, $\Delta_j(n) = \rat{(2^d-1)n}{2^{d-j}} - \rat{(2^d-1)(n-1)}{2^{d-j}} = 2^j + \rat{n-1}{2^{d-j}} - \rat{n}{2^{d-j}}$. The result follows by Lemma~\ref{lem:xy}. \end{proof}

The standard form is not too convenient to work with, mainly because of the floor function. We review the equivalent matrix representation and begin with an example.

\begin{example}\label{ex:rat}
The case $d=4$ was dubbed \emph{fat rat} \cite{Fra2015}; recall the standard form of the $P$-positions without $\0$, $n\ge 1$,
\begin{equation}\label{equation}
\r(n) = \left(\left\lfloor\frac{15}{8}n\right\rfloor, \left\lfloor\frac{15}{4}n\right\rfloor-1, \left\lfloor\frac{15}{2}n\right\rfloor-3, 15n-7\right).
\end{equation}
Let us list the first $11$ expansions of the standard form, and using $t=\lfloor (n-1)/2^{d-1}\rfloor\ge 0$, with $n\ge 1$,
\begin{equation*}
\begin{matrix}
n & r_1(n) & r_2(n) & r_3(n) & r_4(n)\\
1 & 15t+1 & 30t+2 & 60t+4 & 120t+8\\
2 & 15t+3 & 30t+6 & 60t+12 &120t+23\\
3 & 15t+5 & 30t+10 & 60t+19 &120t+38\\
4 & 15t+7 & 30t+14 & 60t+27&120t+53\\
5 & 15t+9 & 30t+17 & 60t+34 &120t+68\\
6 & 15t+11 & 30t+21 & 60t+42 &120t+83\\
7 & 15t+13 & 30t+25 & 60t+49 &120t+98\\
8 & 15t+15 & 30t+29 & 60t+57 & 120t+113\\
9 & 15t+1 & 30t+2 & 60t+4 & 120t+8\\
10 & 15t+3 & 30t+6 & 60t+12 & 120t+23\\
11 & 15t+5 & 30t+10 & 60t+19 & 120t+38
\end{matrix}
\end{equation*}

Notice the periodicity after the first $8$ rows, modulus $15j$, $j=1,2,3,4$ in the respective columns.\footnote{The reader is encouraged to check that the values of $\r(n)$, as $n$ ranges from $1$ to $11$, are identical to the 11 rows of the matrix. For example, for $n=6$, the value of (\ref{equation}) is $(11, 21, 42, 83)$, the same as the line $n=6$, $t=\lfloor 6/8\rfloor=0$ of $\mathcal R_4$. For $n=9$, (\ref{equation} yields $(16, 32, 64, 128)$, same as row 9 of $\mathcal R_4$ with $t = \lfloor 9/8\rfloor = 1$. }

\end{example}

\subsection{In the rat wheel}

Let $\x = (x_1, x_2,\ldots , x_d)$. The periodicity of the standard form, together with the following lemma, motivates us to define $\x \pmod {n}_2$ as the vector $$(x_1\pmod {n}, x_2\pmod {2n}, \ldots , x_d \pmod {2^{d-1}n}).$$

%The periodicity of the standard form, together with the following lemma, motivates us to define $\x \pmod {2^d-1}$ as the vector $$(x_1\pmod {2^d-2^0}, x_2\pmod {2^{d+1}-2^1}, \ldots , x_d \pmod {2^{2d-1}-2^{d-1}}).$$

\begin{lemma}\label{lem:ratwheel}
In the standard form, for all rows $n$, and each column $j$, $$r_{j}(n)\equiv r_{j}(n+ 2^{d-j})\pmod {2^d-1}.$$ Moreover, for all $n$, $\r(n)\equiv \r(n + 2^{d-1})\pmod {2^{d}-1}_2$.  \end{lemma}
\begin{proof} For all $n$, for all $j$,
\begin{align*}
 r_j(n+2^{d-j}) - r_j(n) &= \rat{(2^d-1)(n + 2^{d-j})}{2^{d-j}} - 2^{j-1} + 1 - \left(\rat{(2^d-1)n}{2^{d-j}}-2^{j-1} + 1\right)\\
 &= \rat{(2^d-1)n+2^{d-j}(2^d-1)}{2^{d-j}} - \rat{(2^d-1)n}{2^{d-j}}\\
 &= 2^d-1.
\end{align*}
By maximizing the period (which is always a multiple of 2) at $j = 1$, we obtain the desired periodicity of the rat vectors.
\end{proof}

Since we have this periodic behavior of the rat vectors, for a given number of heaps, it is convenient to represent them in matrix notation. In fact, when we code them modulo $2^{d}-1$, and given the first column, there is a simple bijection with the binary numeration system, which is proved in Theorem~\ref{thm:binrat}, in this section.  This can be seen for any dimension $d$ by studying the saltus and period of the system, as observed in Section~\ref{sec:sequences}. Here, we give a proof of independent interest, using Lemma~\ref{lem:unit}.

\begin{definition}[Binary matrices]
We denote the entry in the $i$th row and the $j$th column of the rat matrix by $R_{i,j}(t)$, where $t\in \N$ is a variable (motivated by Example~\ref{ex:rat}).
Denote the $i$th row, $i\in \{0, \ldots , 2^{d-1}-1\}$, of the \emph{rat-matrix} by $R_i(t)$, $n\in \N_0$. Let $R_{i,1}(t) = (2^{d}-1)t + 2i+1$, and for $j\in \{2,\ldots , d\}$,  $R_{i,j}(t) = 2R_{i,j-1}(t)-b_{i, d-j},$ where $b_i = b_{i,d-2}\cdots b_{i,0}$ is the number $i$ represented in binary.
\end{definition}
Note that, by using the binary representation, it is natural to index the rows of the rat-matrix by $i\in \{0, \ldots , 2^{d-1} - 1\}$ (but in the standard form, we follow the tradition, and start the indexing of rows with $n = 1, 2, \ldots $).

We have the following identity, between consecutive rows in the infinite matrix.
\begin{lemma}[A rat-gap identity]\label{lem:unit}
For any $d\ge  2$, and any $n\ge 2$,
\begin{align}\label{eq:unit}
\sum_{j\in \{2,\ldots ,d\}} 2^{d-j+1}\Delta_{j-1}(n)-2^{d-j}\Delta_j(n) = 1.
\end{align}
\end{lemma}

\begin{proof}
For each row $n\ge 2$, there is a smallest indexed column $\gamma$, such that $n\equiv 0\pmod 2^{d-j}$, and so, by Lemma~\ref{lem:Delta}, $\Delta_\gamma(n) = 2^\gamma - 1$ (note, for all rows, $\Delta_d(n) = 2^d - 1$).  It follows that $\Delta_\rho(n) = 2^{\rho} - 1$, for all $\rho\ge \gamma$. By Lemma~\ref{lem:Delta}, in the expression (\ref{eq:unit}), the powers of 2 will get cancelled, so we are only concerned with the part $-2^{d - \gamma}(-1) + 2^{d-\gamma}(-1)-2^{d-\gamma -1}(-1) + 2^{d-\gamma-1}(-1)-2^{d-\gamma -2}(-1) + \ldots + 2^{d-d+1}(-1)-2^{d-d}(-1)=1$.
\end{proof}

\begin{theorem}[Rats are binary]\label{thm:binrat}
Let $d\ge 2$. For all $n\in \N$, $$\r(n) = R_{\,n-1\!\!\!\pmod {2^{d-1}}}\left({\rat {n-1}{2^{d-1}}}\right).$$
\end{theorem}

\begin{proof}

For each $n\in \N$, we must verify that, with $t = {\rat {n-1}{2^{d-1}}}$, row $n-1$ in the rat-matrix corresponds with row $n$ in the standard form: recall, for all columns $j$, $$r_{j}(n) = \rat{(2^d-1)n}{2^{d-j}}-2^{j-1} + 1.$$

We study the gaps of the  entries in the columns of the respective forms. First we show that they correspond within a rat-matrix, and then we demostrate that the glueing of matrices gives back the infinite form. Recall that $\Delta_{j}(n) = r_{j}(n)-r_j(n-1)$.

Let us begin by showing that row 0 in the rat-matrix corresponds to the first row in the standard form. Since $R_{0,1}(t) = (2^{d}-1)t +1$ and, by $b_0 = \0$, then, for all $j\in \{1,\ldots , d\}$, $R_{0,j}(0) =  2^{j-1}$. Also
\begin{align}
r_{j}(1) &= \rat{2^d-1}{2^{d-j}}-2^{j-1} + 1\\ &= \rat{2^d}{2^{d-j}}+\rat{-1}{2^{d-j}}-2^{j-1} + 1\\ &= \frac{2^d}{2^{d-j}}-1-2^{j-1} + 1\\ &=2^{j-1},
\end{align}
for all $j$.

We want to show that for any row $n$, $\r(n) = R_{i}({\rat {n-1}{2^{d-1}}})$, with $n-1 = \alpha2^{d-1}+i$, for some non-negative integer $\alpha$, and where
\begin{align}\label{eq:ibounds}
0\le i <2^{d-1}.
\end{align}
We begin by showing that the first entries correspond, and note that the last equality in both simplifications follow by (\ref{eq:ibounds}).

\begin{align*}
r_{1}(n)&=\rat{(2^d-1)n}{2^{d-1}}\\
&= \rat{(2^d-1)(\alpha2^{d-1}+i+1)}{2^{d-1}}\\
&= \alpha(2^d-1) + \rat{(2^d-1)(i+1)}{2^{d-1}}\\
&= \alpha(2^d-1) + 2i+2+\rat{-i-1}{2^{d-1}}\\
&= \alpha(2^d-1) + 2i+2-1
\end{align*}

\begin{align*}
R_{i,1}\left(\rat{t-1}{2^{d-1}}\right) &= (2^d-1)\rat{t-1}{2^{d-1}}+2i+1\\
&= (2^d-1)\rat{\alpha 2^{d-1}+i}{2^{d-1}}+2i+1\\
&= (2^d-1)\left(\alpha +\rat{i}{2^{d-1}}\right)+2i+1\\
&= \alpha(2^d-1) + \rat{i}{2^{d-1}}(2^d-1) + 2i + 1\\
&= \alpha(2^d-1) + 2i + 1
\end{align*}

Next, for all $j\in \{2,\ldots , d\}$, we show that
$$b_{i, d-j} = 2R_{i,j-1}(t)-R_{i,j}(t) = 2r_{j-1}(n+1)-r_{j}(n+1).$$
Of course $$b_i-b_{i-1}=\sum 2^{d-j}b_{i,d-j}-\sum 2^{d-j}b_{i-1,d-j} = 1.$$
Hence, it suffices to show that
$$\sum 2^{d-j} (2r_{ j-1}(n+1)-r_{j}(n+1)-(2r_{j-1}(n)-r_j(n))) = 1,$$
that is that
$$\sum 2^{d-j+1}(r_{j-1}(n+1)-r_{j-1}(n))+2^{d-j}(r_j(n)-r_{j}(n+1))=1,$$
that is that $$\sum 2^{d-j+1}\Delta_{j-1}(n+1) - 2^{d-j}\Delta_j(n+1) = 1.$$
This follows by Lemma~\ref{lem:unit}.
\end{proof}

\section{The rats' ternary shortcuts}\label{sec:ternary}
In matrix notation, we will list the proper shortcuts excluding the sequence of rat vectors. Let $\F=(\R-\R)\setminus\R$.

We will see that there is a natural (unique) order of the vectors, $\f\in \F$ by letting, for all $i>0$, $\sum_j f_{i, j}  > \sum_j f_{i-1, j}$, and we will regard $\F$ as this infinite matrix on $d$ columns.
Since $R$ is finite, we define the (proper) \emph{shortcut matrix} $F$ ($F$ for forbidden subtractions), which will also be finite, with $d$ columns, and we prove that it contains $3^{d-1}$ rows, using the natural ternary representations, obtained as a consequence of the binary representation of the rat-matrix.\footnote{Thus the proof gives a bit more information than the statement, and we use our understanding of the structure to find the natural order of the rows in $F$. In this section, we abuse notation and say `shortcut matrix' instead of the somewhat lengthy `proper shortcut matrix'.}

 The $(d-1)$-dimensional vector $\t$ is \emph{ternary} if, for all $0\le j \le d-2$, $t_j\in \{0,1,2\}$.

\begin{definition}\label{lem:uniqueternary}
Index the vectors in the set $\{\r-\r'\not\in \R\mid \r,\r'\in R\}$ in increasing right-to-left lexicographic order\footnote{Row $f_i$ is before row $f_j$ if column $k$ is the rightmost  column where they differ, and then $f_{i,k}<f_{j,k}$}, and let $F$ denote the \emph{shortcut matrix} where $\f_i$ is the $i$th row from the top, and starting with row 0.
\end{definition}

\begin{theorem}\label{thm:shortcut}
The shortcut matrix $F$ contains exactly $3^{d-1}$ distinct rows.
\end{theorem}
\begin{proof}
We show that each  $(d-1)$-dimensional ternary vector $t$, describes precisely one row in the matrix, and then the result follows.
For all $i, k\in \{1,\ldots , 2^{d-1}\}$ and $j\in \{1,\ldots ,d\}$, we have that
\begin{align}\label{eq:recursively}
R_{i+1,j} -R_{k+1,j} = 2(r_j(i) -R_{k,j})+b_{i,j}-b_{k,j},
\end{align}
 where $t_j := b_{i,j}-b_{k,j}+1\in \{0,1,2\}$. Hence, for each pair of rows $i, k$, we define the ternary vector $\t = b_i - b_k + 1$. Not that, given any ternary vector $\t$, it is easy to find two binary vectors such that their difference +1 is $\t$. Now, for each $\t$, there is an equivalence class of pairs of binary vectors, and it is given by $2^u$, where $u$ is the number of 1s in $\t$. Suppose now that we produce the same ternary vector $\t$ in two different ways, say by finding rows such that
\begin{align}\label{eq:same}
\t = b_i - b_j = b_k - b_\ell.
\end{align}
We must show that the two ways to obtain $\t$ results in the same row in $F$, and to this purpose it suffices to show that the last entries  $R_{i+1, d} - R_{j+1, d}$ and $R_{k+1,d} - R_{\ell + 1, d}$, in the two representations will be the same. Observe that $i-j > 0$ if and only if $k-\ell >0$. This gives that the respective differences in the first colums will be the same. Then, because of (\ref{eq:same}), then by (\ref{eq:recursively}), we get the claim for the last column. Thus, the definition of $\t$ gives a unique row vector in the shortcut matrix, and so the number of rows is correct.
\end{proof}

We note that the construction in Theorem~\ref{thm:playable} suggests a similar definition of the shortcut-matrix.
\begin{definition}\label{def:tree}
Let $d\ge 2$. For each $i\in \{1,\ldots ,2^{d-1}\}$, $j\in \{1,\ldots ,d\}$, we construct a tree-structure of depth $d$, where the root has label $(j, x) = (d, i(2^d-1))$. If $x$ is even, then the node $(j,x)$ has one child, labeled $(j-1,x/2)$, and otherwise it has two children labeled $(j-1, (x-1)/2)$ (to the left) and $(j-1, (x+1)/2)$ (to the right). Let $T^d$ denote the family of all such trees, and let  $T^d(n)$ denote the same family, but where each label $(j, x)$ has been replaced with $2^{j-1}(2^{d}-1)n+x$. \end{definition}

\begin{theorem}
Each path in $T^d(n)$, from a leaf  to the root, represents a unique row in the shortcut matrix.
\end{theorem}
\begin{proof}
This follows by Theorem~\ref{thm:playable}.
\end{proof}
We obtain the lexicographic order of the rows in the shortcut matrix by reading the paths left to right and starting with $i=1$, etc.

\subsection{A conjectured algorithm for the last column}\label{sec:lastcol}
The last column of the shortcut matrix satisfies a regular behavior, for increasing dimensions $d$. The number of entries of $k(2^d-1)$, for $k\in \N$, is represented by a sequence of vectors $(\sigma^d)_{d\ge 2}$ of lengths $2^{d-2}$: $$(2), (3,2), (4,3,5,2), (5,4,7,3,8,5,7,2), (6,5,9,4,11,7,10,3,11,8,13,5,12,7,9,2), \ldots $$
	The entries of $\sigma^d$ are defined recursively by $\sigma_1^2 = 2$ and, for $d > 2$, $$\sigma^d_1 = \sigma^{d-1}_1 + 1.$$ For all $1\le j\le 2^d$, $$\sigma^{d+1}_{2j} = \sigma^d_j,$$ and for all $1\le j< 2^d$, $$\sigma^{d+1}_{2j+1} = \sigma^d_j+\sigma^d_{j+1}.$$
	
\section{Rat games are approximately nim}\label{sec:approxnim}
The game of nim is probably the most famous impartial combinatorial game. It has the property that any impartial game $G$ is equivalent to a heap of nim; the size of a nim heap is its nim value (a.k.a Sprague-Grundy value). We say that an impartial heap game $\Gamma$ (on $d$ heaps) is \emph{almost nim} if the total number of objects in the heaps is almost always its nim value. Precisely, let $\gamma(n)$ denote the number of game positions with a total number of $n$ objects, for which the nim-value of $\Gamma$ is not $n$. Then $\Gamma$ is \emph{almost-nim-heap} if $$\lim_{n\rightarrow \infty}\frac{\gamma(n)}{n} = 0.$$
Moreover, if the nim value 0 is the only nim value which differs from the total number of objects in the $d$ heaps, then we call $\Gamma$ an \emph{approximate-nim-heap}.

In Table~\ref{tab:SG2heap}, we give as example the Sprague-Grundy values of the rat game on 2 heaps.
\begin{theorem}
The grandiose rat games are almost-nim-heaps. In fact, they are approximate-nim-heaps.
\end{theorem}
\begin{proof}
By construction, the rat vectors have nim value 0. Using the methods in this paper, one can verify that the other heap postions have nim values equal to the total number of objects in the heaps, respectively.
\end{proof}

\begin{table}[ht!]
\centering
\resizebox{.4\textwidth}{!}{%
\begin{tabular}{lllllll}
0 & 1 & 2  & 3  & 4  & 5  & 6  \\
1 & 2 & 3  & 4  & 5  & 6  & 7  \\
2 & 0 & 4  & 5  & 6  & 7  & 8  \\
3 & 4 & 5  & 6  & 7  & 8  & 9  \\
4 & 5 & 6  & 7  & 8  & 9  & 10 \\
5 & 6 & 7  & 0  & 9  & 10 & 11 \\
6 & 7 & 8  & 9  & 10 & 11 & 12 \\
7 & 8 & 9  & 10 & 11 & 12 & 13 \\
8 & 9 & 10 & 11 & 0  & 13 & 14
\end{tabular}%
}
\caption{The Sprague-Grundy values for the rat game on 2 heaps. The North-East corner is the terminal position $(0,0)$.}
\label{tab:SG2heap}
\end{table}

That is, if you play a grandiose rat game in disjunctive sum with another game, then you can play approximately as if the game were nim, just keep in mind the exception that the rat vectors have value zero.\footnote{Note that Singmaster's result \cite{Sing}, that almost no positions in an impartial game are P-positions, implies that any approximate nim-heap is also almost-heap.} This result obviously also holds for the succinct rat games, since we proved that the moves are the same. %This result obviously also holds for the succinct rat games, since all moves are the same; in Section~\ref{sec:playable}, we showed only P-equivalence, but in fact they are equivalent (with respect to nim-values).

\begin{example}
Let $\mathcal G=(1,4,5)_\text{\sc nim}$ be a game of {\sc nim}, and let $\mathcal H=(1,4,5)_\text{\sc rat}$ be a {\sc rat}-game. In the game $\mathcal G + \mathcal H$, a winning move is to $\mathcal G + (1,2,4)_\text{\sc rat}$, since the nim-value of each component game is $0$, and since $\mathcal H - (1,2,4)_\text{\sc rat} = (0,2,1)_\text{\sc rat}\not \in \R$. 

Let $\mathcal G=(1,2,5,8)_\text{\sc nim}$ be a game of {\sc nim}, and let $\mathcal H=(3,4,5,6)_\text{\sc rat}$ be a {\sc rat}-game. In the game $\mathcal G + \mathcal H$, a winning move is to $\mathcal G + (3,0,5,6)_\text{\sc rat}$, since the nim-value of each component game is $14$.

Let $\mathcal G=(1,2,5,8)_\text{\sc nim}$ be a game of {\sc nim}, and let $\mathcal H=(11,21,42,83)_\text{\sc rat}$ be a {\sc rat}-game. In the game $\mathcal G + \mathcal H$, a winning move is to $(1,2,5,6)_\text{\sc nim} + (11,21,42,83)_\text{\sc rat}$, since the nim-value of each component game is $0$. Indeed, $83\equiv 8 \pmod {2^4-1}$, and the recursive word is binary, namely $b_5=101$. 
\end{example}

\section{The rats' right shifts}\label{sec:rightshift}
%In binary representation, the multiples of $2^d-1$ are all numbers of the form $k\times 11\ldots 1$, i.e.
%$3,6,9,12,15,18,21,24, \ldots $, $11, 110, 1001, 1100, 1111,  10010, 10101, 11000, 11011, 11101, 1000001, $ Multiplying a binary number with 2 is a left shift of the number (and adjoining a 0 to the left). 

For $x = x_n\cdots x_0$ a nonnegative integer coded in binary (i.e. $x=\sum 2^ix_i$, $i\in \{0,\ldots , n\}$), let $\varphi(x) = x_n\cdots x_1$ be the binary digits right shift of $x$, where the rightmost digit, $x_0$, has been dropped  (i.e. $\varphi(x) = \sum 2^{i-1}x_i$, $i\in \{1,\ldots , n\}$). Note that $\varphi (x) = \lfloor x/2 \rfloor$. Let $\xi (x) = x_0+1\pmod 2$ be the binary complement to the dropped digit. We have the following result: 

\begin{prop}
%Let $\y$ = \x \pmod{2^d - 1}$. 
Coded in binary, let $\alpha = \xi(x_2+1)\cdots \xi(x_d+1)$ (i.e. $\alpha = \sum 2^{i-2}\xi(x_i+1)$, $i\in \{2,\ldots , d\}$). The vector $\x = (x_1, \ldots , x_d)\in \R$ if and only if $\xi(x_i + 1)\in \{0, 1\}$ and $\varphi(x_i + 1) = x_{i - 1}$, for all $i$, and $\alpha(2^d-1) \equiv  x_d - 2^{d-1} \pmod {2^{d-1}(2^d-1)}$. 
\end{prop}

\begin{proof}
This is just a reformulation of previous results, in particular Lemma~\ref{lem:ratwheel} and Theorem~\ref{thm:binrat}.
\end{proof}

A similar, but weaker result can be obtained for the shortcut matrix, but we omit it here, since the nice correspondence with the row numbers does not hold any more. Instead the characterization depends on understanding the general line fractals displayed in the Appendix. %of the arXiv manuscript \cite{RatData}.
%\section{A right shift property}
%From position $\x$, the 
%\iffalse
\section*{Appendix}%\label{sec:appendix}
Beginning with $d=2$, we have the standard form $(\rat{3n}{2}, 3n-1)$, and for $n\in \mathbb N_0$, the mouse-matrix gives all non-zero P-positions. (Here the variable $n$ has different interpretations in the standard forms and the matrices.)
\begin{equation*}
R_2=\left(
\begin{matrix}
3n+1 & 6n+2 \\
3n+3 & 6n+5
\end{matrix}
\right) \qquad
\end{equation*}
The mouse's shortcut-matrix consists of all vector differences of $R_2$.
\begin{equation*}
F_2 =\left(
\begin{matrix}
3n & 6n \\
3n+1 & 6n+3 \\
3n+2 & 6n+3\\
\end{matrix}
\right) \qquad
\end{equation*}
The standard form for $d=3$ is  $(\rat{7n}{4}, \rat{7n}{2}-1, 7n-3)$.
The rat-matrix is
\begin{equation*}
R_3=\left(
\begin{matrix}
7n+1 & 14n+2 & 28n+4 \\
7n+3 & 14n+6 & 28n+11 \\
7n+5 & 14n+9 & 28n+18 \\
7n+7 & 14n+13 & 28n+25
\end{matrix}
\right)\qquad
\end{equation*}

and its shortcut-matrix is
\begin{equation*}
F_3 = \left(
\begin{matrix}
7n & 14n & 28n \\
7n+1 & 14n+3 & 28n+7 \\
7n+2 & 14n+3 & 28n+7 \\
7n+2 & 14n+4 & 28n+7\\
7n+3 & 14n+7 & 28n+14 \\
7n+4 & 14n+7 & 28n+14 \\
7n+5 & 14n+10 & 28n+21\\
7n+5 & 14n+11 & 28n+21 \\
7n+6 & 14n+11 & 28n+21
\end{matrix}
\right) \qquad
\end{equation*}
In the ternary recurrence we use a recursive division by 2, beginning with the last column, and note if the result is exact (1), smaller (0) or larger (2); here indicated in a `difference'  matrix accompanying $F_3$:

\begin{equation*}
\left(
\begin{matrix}

2 & 2  \\
0 & 2  \\
1 & 0 \\
2 & 1 \\
0 & 1  \\
1 & 2 \\
2 & 0  \\
0 & 0

\end{matrix}
\right) \qquad
\end{equation*}

Interpreting these numbers in ternary and noting the differences between the consecutive rows produces pictures in the Appendix (for $d=2,\ldots 10$). 

The standard form for the fat rat is
$$\r_4=\left(\left\lfloor\frac{15}{8}n\right\rfloor, \left\lfloor\frac{15}{4}n\right\rfloor-1, \left\lfloor\frac{15}{2}n\right\rfloor-3, 15n-5\right)$$
with matrix
\begin{equation*}
R_4 = \left(
\begin{matrix}
15n+1 & 30n+2 & 60n+4 & 120n+8\\
15n+3 & 30n+6 & 60n+12 &120n+23\\
15n+5 & 30n+10 & 60n+19 &120n+38\\
15n+7 & 30n+14 & 60n+27&120n+53\\
15n+9 & 30n+17 & 60n+34 &120n+68\\
15n+11 & 30n+21 & 60n+42 &120n+83\\
15n+13 & 30n+25 & 60n+49 &120n+98\\
15n+15 & 30n+29 & 60n+57 & 120n+113
\end{matrix}
\right)\qquad
\end{equation*}

\begin{equation*}
F_4 = \left(
\begin{matrix}
15n & 30n & 60n &120n\\
15n+1 & 30n+3 & 60n+7 &120n+15\\
15n+2 & 30n+3 & 60n+7 & 120n+15\\
15n+2 & 30n+4 & 60n+7 & 120n+15\\
15n+2 & 30n+4 & 60n+8 & 120n+15\\
15n+3 & 30n+7 & 60n+15 &120n+30\\
15n+4 & 30n+7 & 60n+15 & 120n+30\\
15n+4 & 30n+8 & 60n+15 &120n+30\\
15n+5 & 30n+11 & 60n+22 &120n+45\\
15n+6 & 30n+11 & 60n+22 & 120n+45\\
15n+5 & 30n+11 & 60n+23 &120n+45\\
15n+6 & 30n+11 & 60n+23 &120n+45\\
15n+6 & 30n+12 & 60n+23 &120n+45\\
15n+7 & 30n+15 & 60n+30 &120n+60\\
15n+8 & 30n+15 & 60n+30 & 120n+60\\
15n+9 & 30n+18 & 60n+37 &120n+75\\
15n+9 & 30n+19 & 60n+37 &120n+75\\
15n+10 & 30n+19 & 60n+37 &120n+75\\
15n+9 & 30n+19 & 60n+38 & 120n+75\\
15n+10 & 30n+19 & 60n+38 &120n+75\\
15n+11 & 30n+22 & 60n+45 & 120n+90\\
15n+11 & 30n+23 & 60n+45 & 120n+90\\
15n+12 & 30n+23 & 60n+45 &120n+90\\
15n+13 & 30n+26 & 60n+52 & 120n+105\\
15n+13 & 30n+26 & 60n+53 & 120n+105\\
15n+13 & 30n+27 & 60n+53 & 120n+105\\
15n+14 & 30n+27 & 60n+53 &120n+105
\end{matrix}
\right) \qquad
\end{equation*}
Here we illustrate the `difference matrix' leading up to the ternary recurrence, case $d=4$:
\begin{equation*}
\left(
\begin{matrix}
2 & 2 & 2 \\
0 & 2& 2 \\
1 & 0 & 2 \\
1 & 1 & 0 \\
2 & 2 & 1 \\
0 & 2 & 1 \\
1 & 0 & 1 \\
2 & 1 & 2 \\
0 & 1 & 2 \\
2 & 2 & 0 \\
0 & 2 & 0 \\
1 & 0 & 0 \\
2 & 1 & 1 \\
0 & 1 & 1 \\
1 & 2 & 2 \\
2 & 0 & 2 \\
0 & 0 & 2 \\
2 & 1 & 0 \\
0 & 1 & 0 \\
1 & 2 & 1 \\
2 & 0 & 1 \\
0 & 0 & 1 \\
1 & 1 & 2 \\
1 & 2 & 0 \\
2 & 0 & 0 \\
0 & 0 & 0 \\
\end{matrix}
\right) \qquad
\end{equation*}

\subsection*{Fractals in the shortcuts}%\label{sec:anarchy}
We have showed that games for grand rats are well behaved, in fact, they are playable. There is a way to capture the full behavior of the associated ternary matrices as $d$ grows. When we code the rows lexicographically then the ternary recurrence satisfies a 2-dimensional Cantor-like line-fractal, shown in some pictures below for small $d$. The full characterization of these matrices can be recursively defined, without mention of their definition via the short-cut matrices. In Section~\ref{sec:lastcol}, we included a conjectured formula for the last column of the shortcut matrix. At the end of this Appendix, we include some data for the respective `base lines' in the pictures. The remaining line-fractals adapt this patterns, translated to various (Cantor-like) positions and lengths. The penultimate base-line consists of two copies the baseline of the previous picture. The top line appears only from $d\ge 4$ is a linear translation of the penultimate line from the previous picture. The penultimate top line has a somewhat similar characteristicts as the penultimate bottom line; it consists of two copies of the previous penultimate topline. The uppermost points in the cases $d=2,3$ are in fact the penultimate topline (in this sense). The rest of the pictures consists of copies of the line fractals of previous pictures. The detailed description of this is fairly technical, so we omit it in this study. When completed, we also aim for a (computer aided) proof of consistency of the recursive construction with the definition of the ternary matrixes (using also the conjectured recurrence of the last column). 

\begin{figure}[ht!]
  \centering
   \caption{Shortcut differences for $d=2,3,4,5,6$ respectively.}\vspace{.5 cm}
     \includegraphics[width=.45\textwidth]{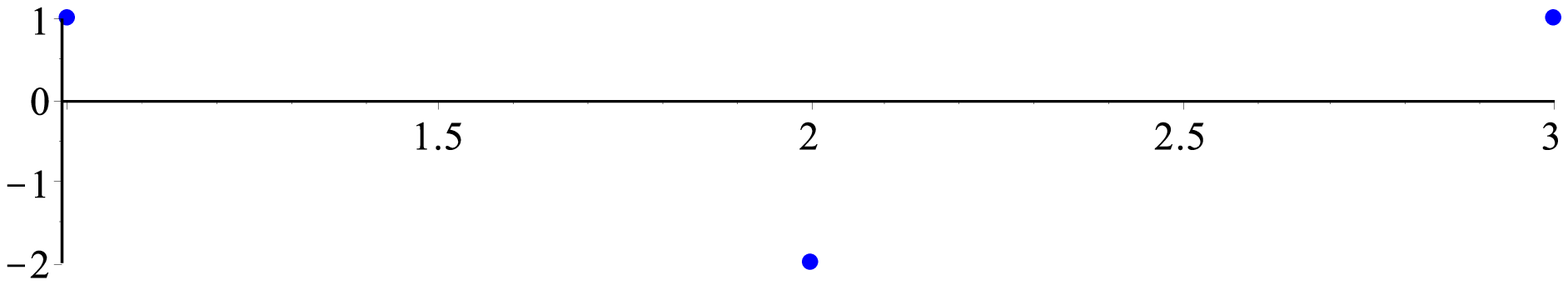}\hspace{1 cm}
     \includegraphics[width=.45\textwidth]{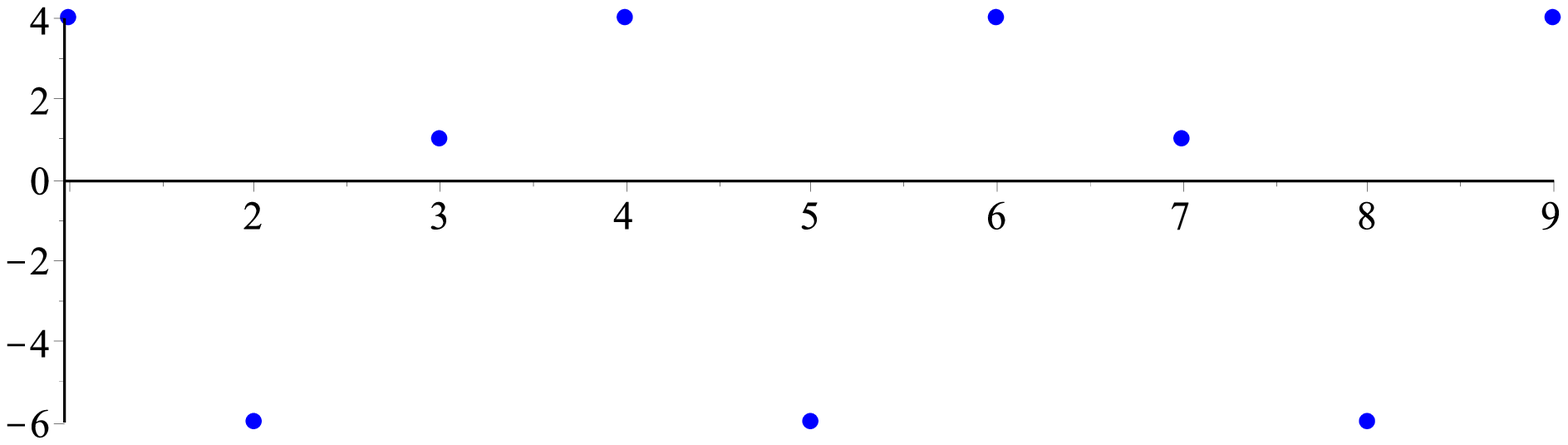}\vspace{1 cm}

    \includegraphics[width=.45\textwidth]{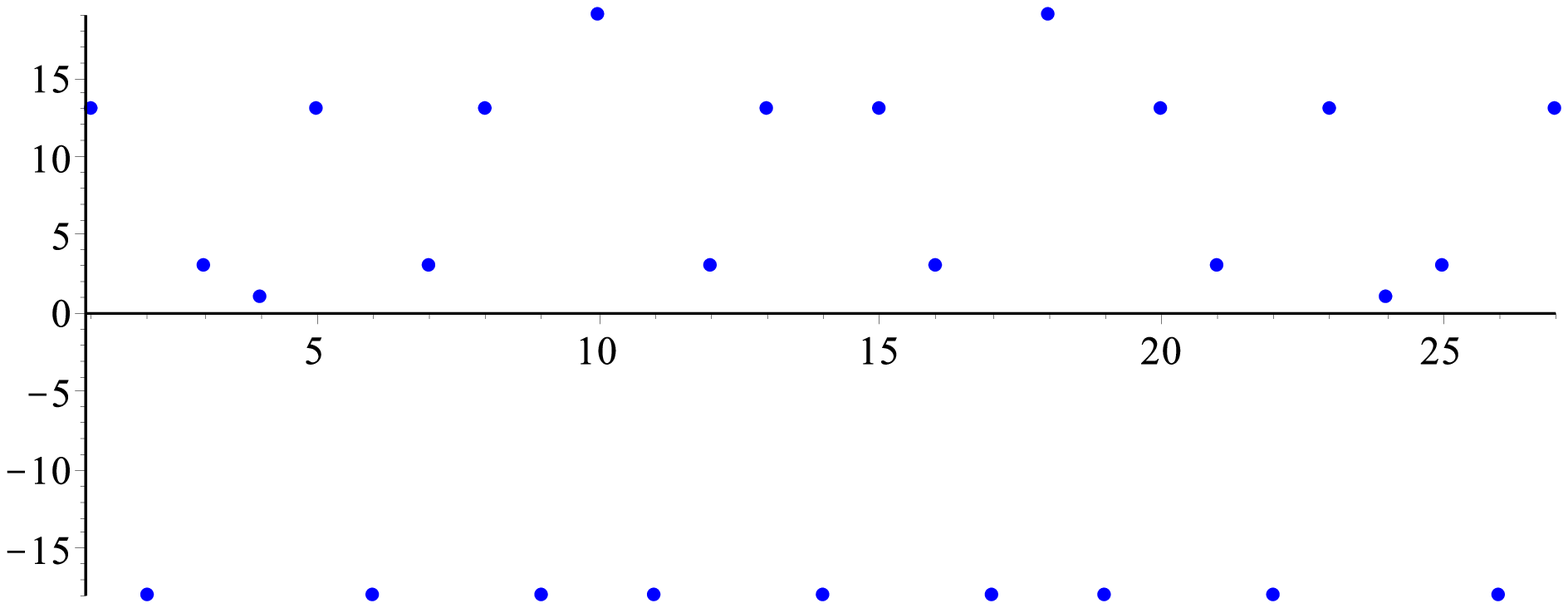}\hspace{1 cm}
     \includegraphics[width=.45\textwidth]{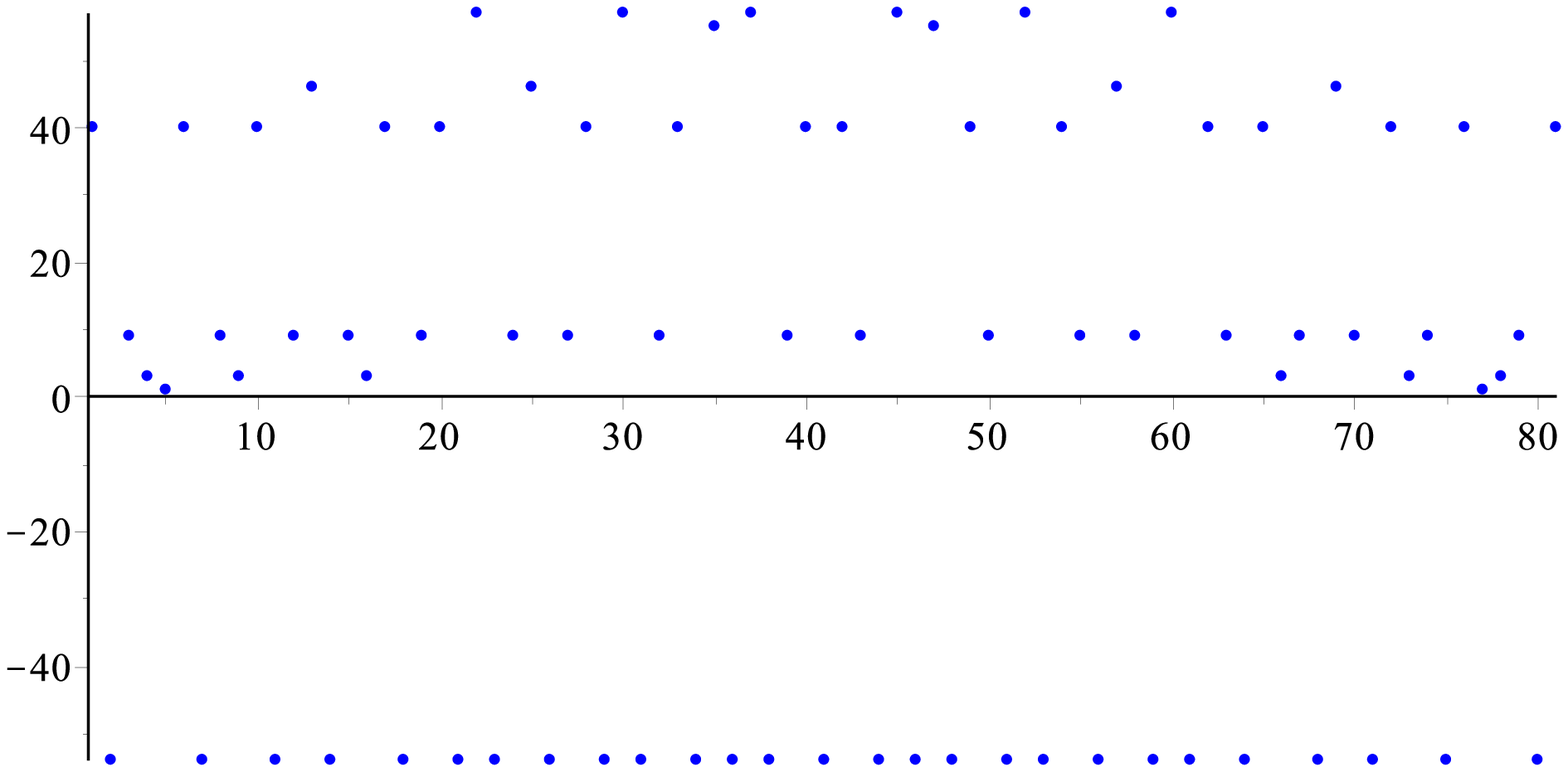}\vspace{1.2 cm}
      
     \includegraphics[width=.8\textwidth]{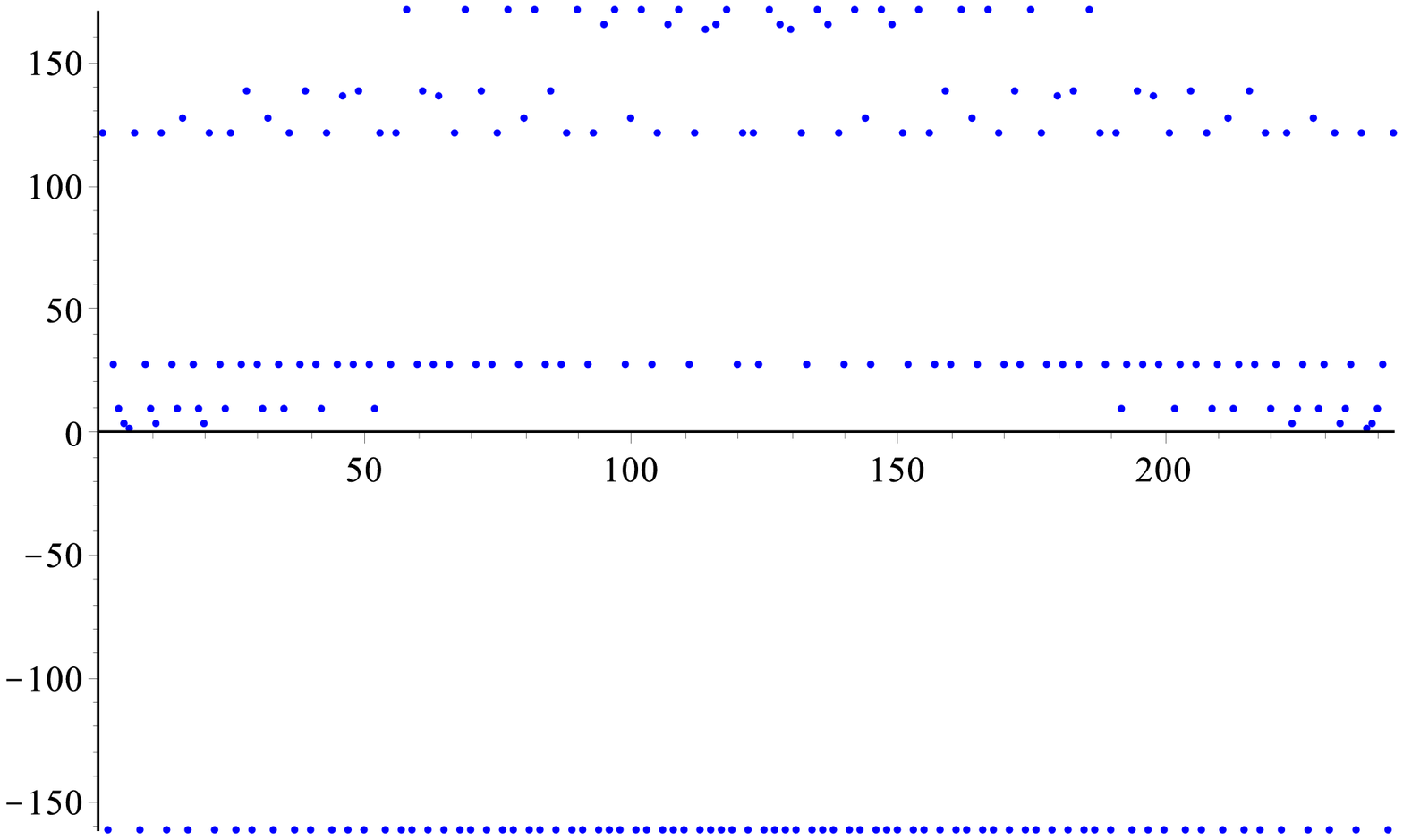}
 \end{figure}
\begin{figure}[ht!]
  \centering
    \caption{Shortcut differences for $d = 7,8$, for values $\ge 0$.}\vspace{.5 cm}
        
    \includegraphics[width=.85\textwidth]{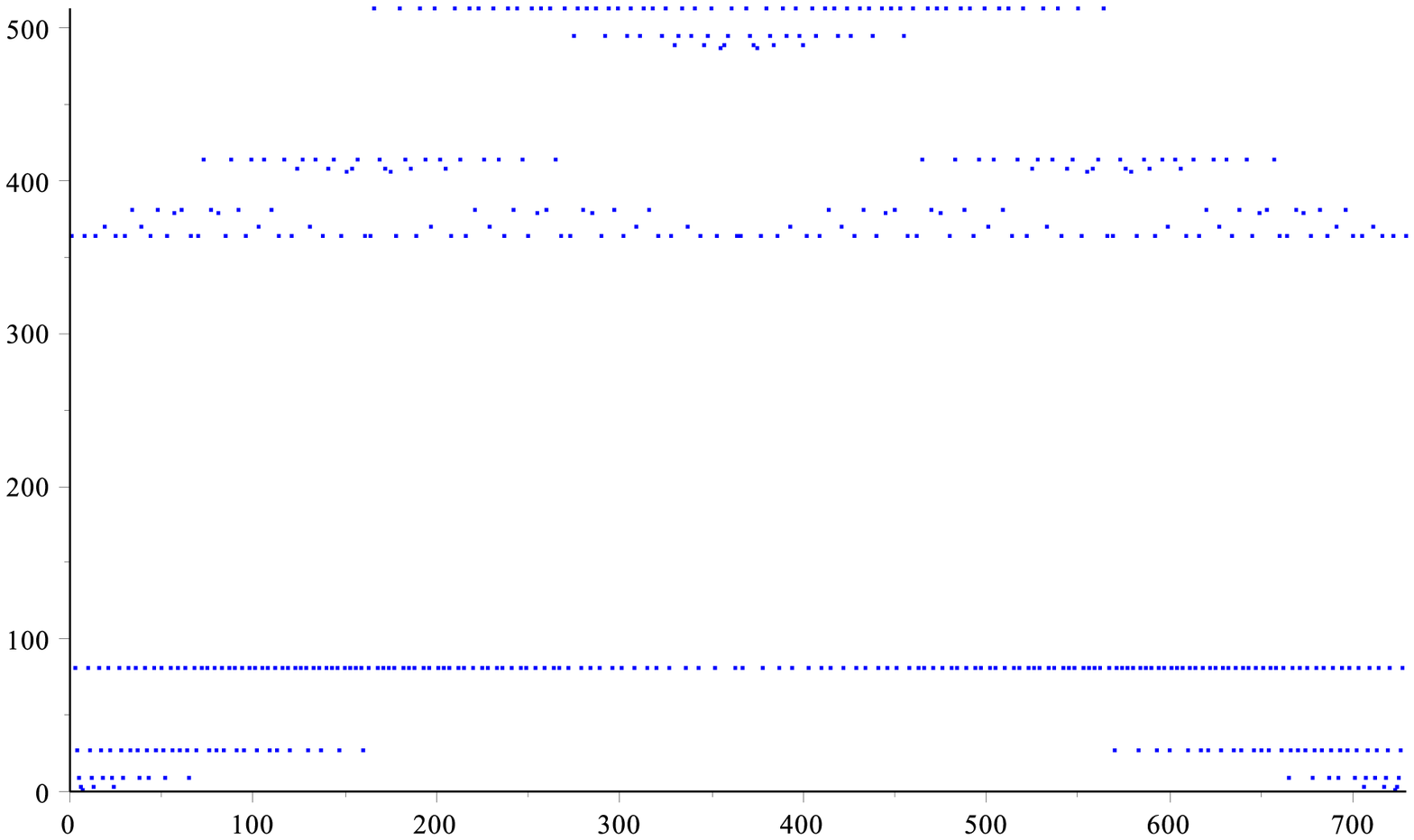}\vspace{10 mm}
    
    \includegraphics[width=.85\textwidth]{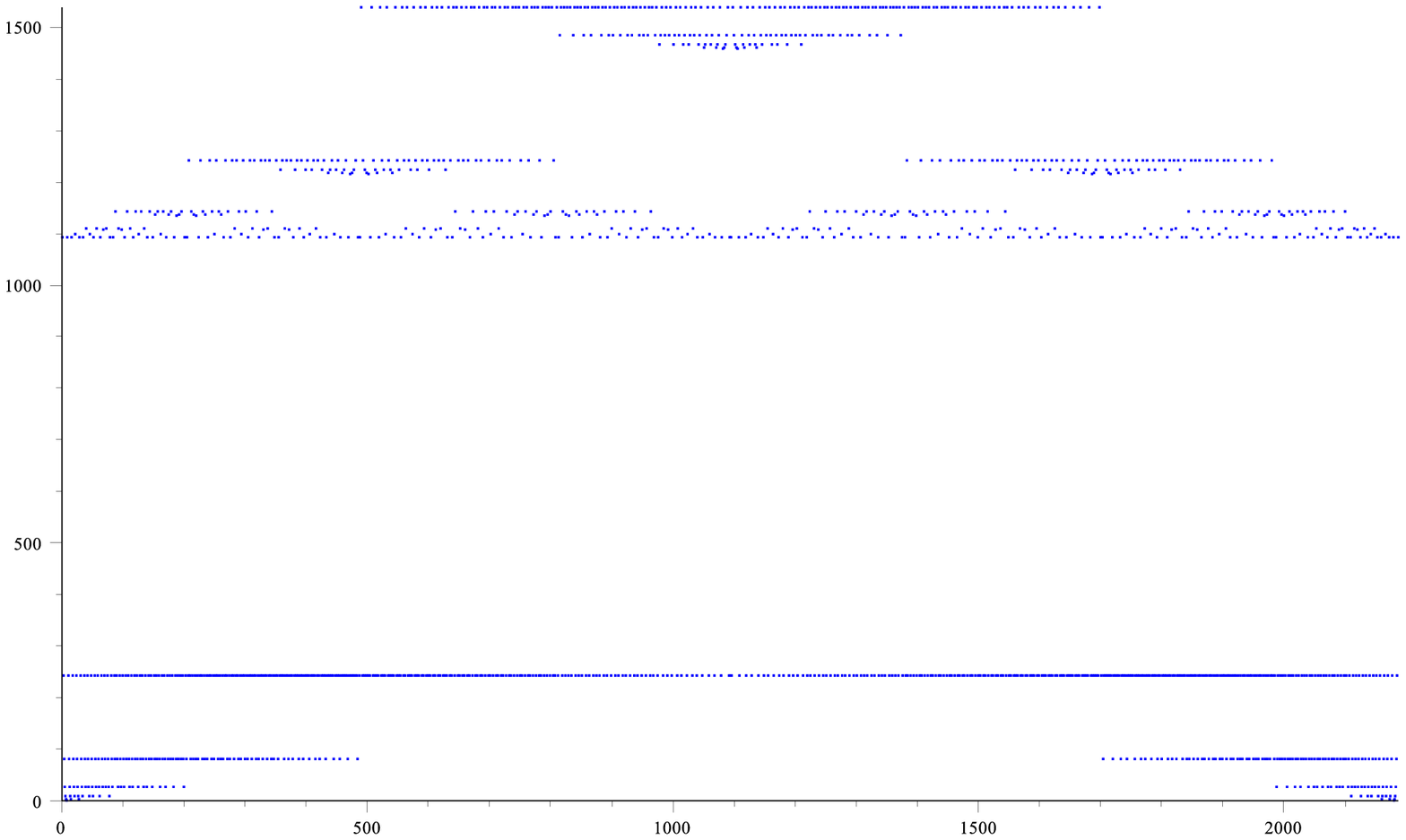}\
    \end{figure}
\begin{figure}[ht!]
  \centering
            \caption{Shortcut differences for $d=9,10$, for values $\ge 0$.}\vspace{7 mm}
    \includegraphics[width=.8\textwidth]{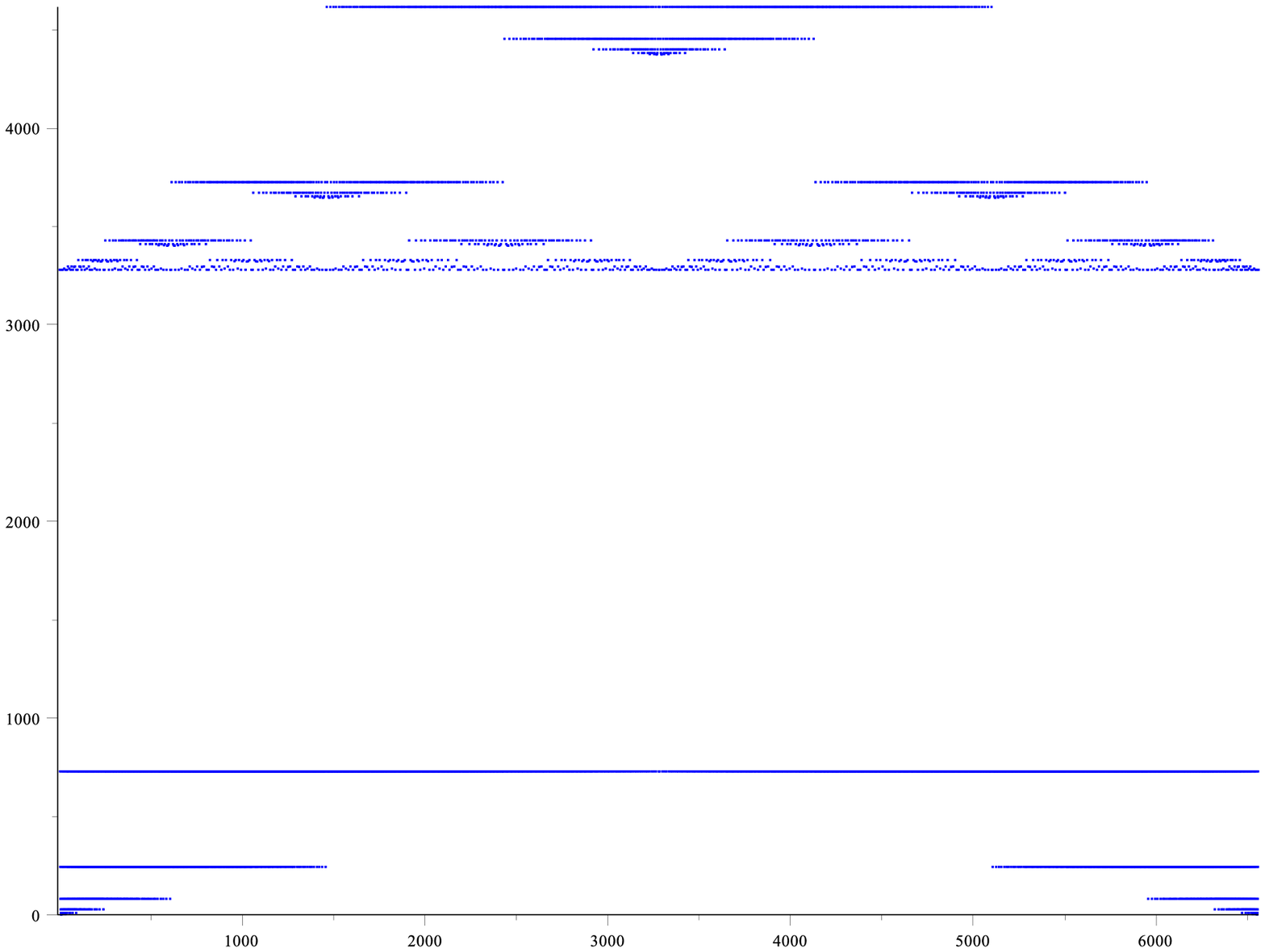}\vspace{10 mm}
    
       \includegraphics[width=.8\textwidth]{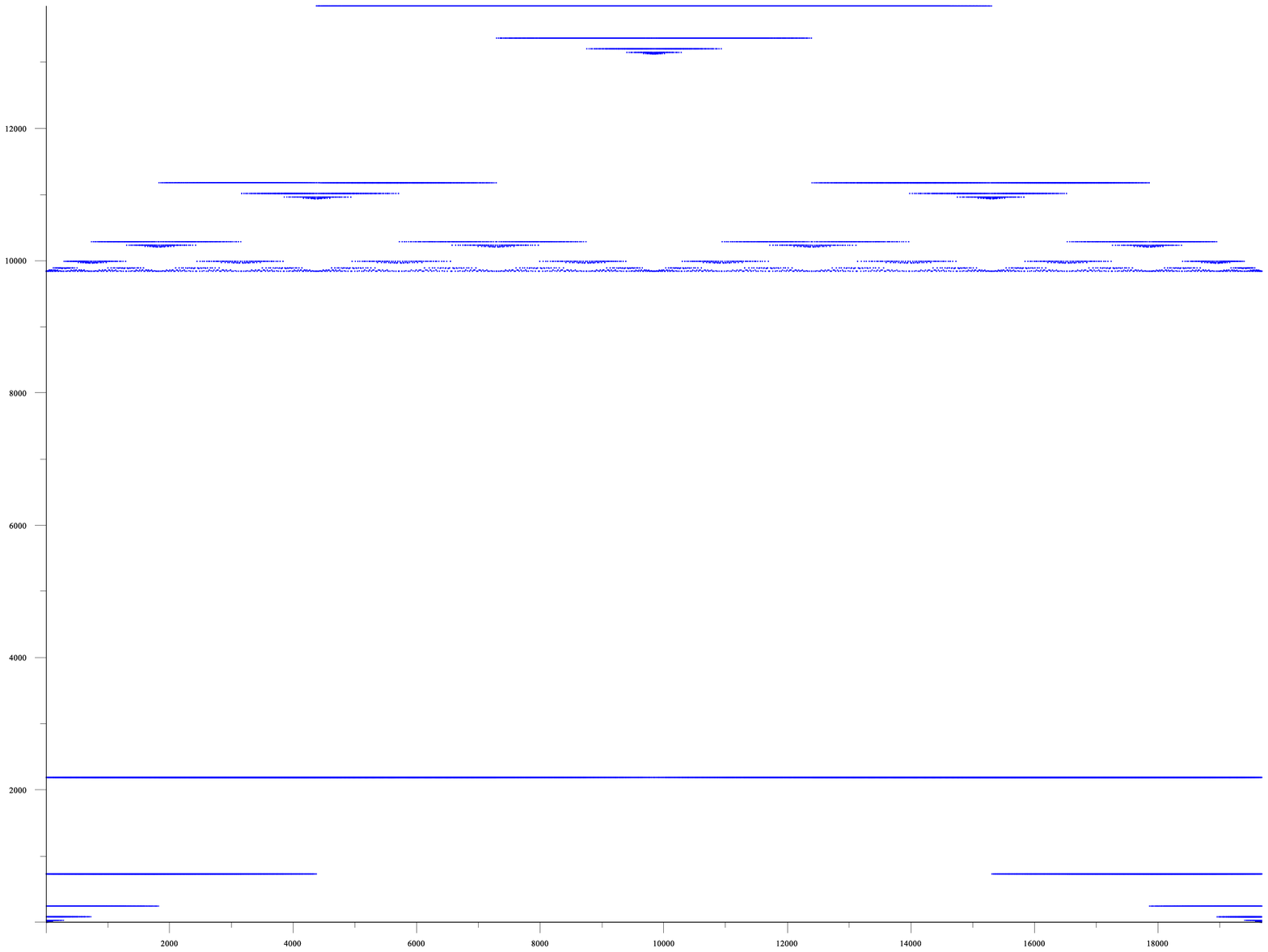}
\end{figure}

\clearpage
The sets of iterated differences of the shortcut matrix, using ternary recurrence, decoded in base 3 expansion ($-1\rightarrow 0, 0\rightarrow 1, 1\rightarrow 2$), are [in brackets the corresponding number of occurences of the numbers]:\\

$d=2:$ -2,1: [1, 2]\\

$d=3:$ -6,1,4: [3, 2, 4]\\

$d=4:$ -18, 1, 3, 13, 19: [9, 2, 6, 8, 2]\\

$d=5:$ -54, 1, 3, 9, 40, 46, 55, 57: [27, 2, 6, 18, 16, 4, 2, 6]\\

$d=6:$ -162, 1, 3, 9, 27, 121, 127, 136, 138, 163, 165, 171: [81, 2, 6, 18, 54, 32, 8, 4, 12, 2, 6, 18]\\

$d=7:$ -486, 1, 3, 9, 27, 81, 364, 370, 379, 381, 406, 408, 414, 487, 489, 495, 513: [243, 2, 6, 18, 54, 162, 64, 16, 8, 24, 4, 12, 36, 2, 6, 18, 54]\\

$d=8:$-1458, 1, 3, 9, 27, 81, 243, 1093, 1099, 1108, 1110, 1135, 1137, 1143, 1216, 1218, 1224, 1242, 1459, 1461, 1467, 1485, 1539: [729, 2, 6, 18, 54, 162, 486, 128, 32, 16, 48, 8, 24, 72, 4, 12, 36, 108, 2, 6, 18, 54, 162]\\

$d=9$:
-4374, 1, 3, 9, 27, 81, 243, 729, 3280, 3286, 3295, 3297, 3322, 3324, 3330, 3403, 3405, 3411, 3429, 3646, 3648, 3654, 3672, 3726, 4375, 4377, 4383, 4401, 4455, 4617: [2187, 2, 6, 18, 54, 162, 486, 1458, 256, 64, 32, 96, 16, 48, 144, 8, 24, 72, 216, 4, 12, 36, 108, 324, 2, 6, 18, 54, 162, 486]\\

$d=10$: -13122, 1, 3, 9, 27, 81, 243, 729, 2187, 9841, 9847, 9856, 9858, 9883, 9885, 9891, 9964, 9966, 9972, 9990, 10207, 10209, 10215, 10233, 10287, 10936, 10938, 10944, 10962, 11016, 11178, 13123, 13125, 13131, 13149, 13203, 13365, 13851: [6561, 2, 6, 18, 54, 162, 486, 1458, 4374, 512, 128, 64, 192, 32, 96, 288, 16, 48, 144, 432, 8, 24, 72, 216, 648, 4, 12, 36, 108, 324, 972, 2, 6, 18, 54, 162, 486, 1458]\\

$d=11$: -39366, 1, 3, 9, 27, 81, 243, 729, 2187, 6561, 29524, 29530, 29539, 29541, 29566, 29568, 29574, 29647, 29649, 29655, 29673, 29890, 29892, 29898, 29916, 29970, 30619, 30621, 30627, 30645, 30699, 30861, 32806, 32808, 32814, 32832, 32886, 33048, 33534, 39367, 39369, 39375, 39393, 39447, 39609, 40095, 41553\\

$d=12$: -118098 (baseline), 1, 3, 9, 27, 81, 243, 729, 2187, 6561(pb), 19683(pt),$\star 88573$, @88579, @88588, 88590, @88615, 88617, 88623, @88696, 88698, 88704, 88722, @88939, 88941, 88947, 88965, 89019, @89668, 89670, 89676, 89694, 89748, 89910, @91855, 91857, 91863, 91881, 91935, 92097, 92583, @ 98416, 98418, 98424, 98442, 98496, 98658, 99144, 100602, @118099, 118101, 118107, 118125, 118179, 118341, 118827, 120285, 124659\\

Observations: The number of elements in the sets of differences are $2,3,5,8,12,17,23,30,38,47,57...$, which is $\xi=\xi_d=2+\left({d-1} \atop {2}\right) = \frac{d^2 - 3d + 3}{2}$. The smallest value is $-a(d)$, where $a(d) = 2\cdot3^{d-2}$, and the largest value is $a(d)+3^{d-4}$, for $d\ge 4$. The 2nd to the $(d-1)$th numbers are $1,\ldots , 3^{d-3}$. The next entry is given by the sequence $1,4,13,40,121,364,1093,\ldots$, and it is $(3^{d - 1}-1)/2$. The next entry is $(3^{d - 1}-1)/2 + 6$ and then $(3^{d - 1}-1)/2+15$, $(3^{d - 1}-1)/2+17$, and so on.

We describe the elements in explicit formulas depending only on $d$ (the fractals that imitate the first difference sequence of fractals for $d=2,3,\ldots$ are indicated with @ symbol in case $d=12$. The $\star$ is where the fractals start. The line at (pt) is the penultimate top line. 

The beginning is: $\tau(0) = - 2\cdot3^{d-2}, \tau(1,d-2) = 3^i, \tau(d-1) = (3^{d - 1}-1)/2, \tau(d) = (3^{d - 1}-1)/2 + 6$. Here the sequence of sequences starts; each new sequence has one more element than the previous. The starting values are $\tau_i = \tau_{i-1} + 3^{i+1}$, for $i\in \{1, \ldots , d-4\}$ with $\tau_0 = \tau(d)$. Here: for $d = 3$, the sequence of sequences does not get started; for $d = 4$, it has one element; for $d > 4$ the starting values are defined. Now we define the sequneces. For  $d > 4$, he $i$th sequence is $\tau^i = \tau_i-1+3^0,\ldots ,\tau_i-1+3^{i}$. To check: the number of elements in the ith sequence is $i$; the final sequence is $x_{\xi-d+3+i} = a(n)+3^i$, for all $i\in \{0,\ldots ,d-4\}$.

The number of representatives of each element in the set satisfies a similar description as the $\tau$ sequences, and is indicated with []-brackets above for all $d\le 10$.

\begin{remark}
We observe that if one uses instead the reverse lexicographical order (building the ternary matrix from smaller instead of from larger indices), then we get similar, but more complicated ternary fractals. In this case, as far as we have been able to compute, the number of elements in the sets $2,3,5, \ldots$ satisfy SLOANE: https://oeis.org/A011826 :	$f$-vectors for simplicial complexes of dimension at most 1 (graphs) on at most $d-2$ vertices (prof. Svante Linusson, KTH, Stockholm). $a(d) = ((d-1)^3 - 3(d-1)^2 + 8(d-1) + 6)/6$. It is equal to the number of compressed (lexicographic order) simplicial complexes on $d-2$ variables and word length at most 2. Perhaps the fractals we study here also have interesting connections to simplicial complexes?
\end{remark}

\subsection*{Data}%\label{sec:data}
The base set is at $-2\cdot 3^{d-2}$ (baseline), differences between occurences; this set is obtained as the complement of the other elements at the same level, and then it is used for the next level to give the penultimate base line (pb): \\

\noindent[]\\

\noindent[3, 3]\\

\noindent[4, 3, 2, 3, 3, 2, 3, 4]\\

\noindent[5, 4, 3, 4, 3, 2, 3, 3, 2, 3, 2, 2, 3, 3, 2, 2, 3, 2, 3, 3, 2, 3, 4, 3, 4, 5]\\

\noindent[6, 5, 4, 5, 4, 3, 4, 4, 3, 4, 3, 3, 4, 3, 2, 3, 3, 3, 2, 3, 3, 2, 3, 2, 3, 3, 2, 3, 2, 2, 3, 2, 3, 2, 2, 3, 2, 2, 2, 3, 3, 2, 2, 2, 3, 2, 2, 3, 2, 3, 2, 2, 3, 2, 3, 3, 2, 3, 2, 3, 3, 2, 3, 3, 3, 2, 3, 4, 3, 3, 4, 3, 4, 4, 3, 4, 5, 4, 5, 6]\\

\noindent[7, 6, 5, 6, 5, 4, 5, 5, 4, 5, 4, 4, 5, 4, 3, 4, 4, 4, 3, 4, 4, 3, 4, 3, 4, 4, 3, 4, 3, 3, 4, 3, 4, 3, 3, 4, 3, 3, 3, 4, 3, 2, 3, 3, 3, 3, 2, 3, 3, 3, 2, 3, 3, 2, 3, 3, 3, 2, 3, 3, 2, 3, 2, 3, 3, 2, 3, 3, 2, 3, 2, 3, 3, 2, 3, 2, 3, 2, 3, 3, 2, 3, 2, 2, 3, 2, 3, 2, 3, 2, 2, 3, 2, 3, 2, 2, 3, 2, 2, 3, 2, 3, 2, 2, 3, 2, 2, 2, 3, 2, 2, 3, 2, 2, 2, 3, 2, 2, 2, 2, 3, 3, 2, 2, 2, 2, 3, 2, 2, 2, 3, 2, 2, 3, 2, 2, 2, 3, 2, 2, 3, 2, 3, 2, 2, 3, 2, 2, 3, 2, 3, 2, 2, 3, 2, 3, 2, 3, 2, 2, 3, 2, 3, 3, 2, 3, 2, 3, 2, 3, 3, 2, 3, 2, 3, 3, 2, 3, 3, 2, 3, 2, 3, 3, 2, 3, 3, 3, 2, 3, 3, 2, 3, 3, 3, 2, 3, 3, 3, 3, 2, 3, 4, 3, 3, 3, 4, 3, 3, 4, 3, 4, 3, 3, 4, 3, 4, 4, 3, 4, 3, 4, 4, 3, 4, 4, 4, 3, 4, 5, 4, 4, 5, 4, 5, 5, 4, 5, 6, 5, 6, 7]\\

\noindent[8, 7, 6, 7, 6, 5, 6, 6, 5, 6, 5, 5, 6, 5, 4, 5, 5, 5, 4, 5, 5, 4, 5, 4, 5, 5, 4, 5, 4, 4, 5, 4, 5, 4, 4, 5, 4, 4, 4, 5, 4, 3, 4, 4, 4, 4, 3, 4, 4, 4, 3, 4, 4, 3, 4, 4, 4, 3, 4, 4, 3, 4, 3, 4, 4, 3, 4, 4, 3, 4, 3, 4, 4, 3, 4, 3, 4, 3, 4, 4, 3, 4, 3, 3, 4, 3, 4, 3, 4, 3, 3, 4, 3, 4, 3, 3, 4, 3, 3, 4, 3, 4, 3, 3, 4, 3, 3, 3, 4, 3, 3, 4, 3, 3, 3, 4, 3, 3, 3, 3, 4, 3, 2, 3, 3, 3, 3, 3, 2, 3, 3, 3, 3, 2, 3, 3, 3, 2, 3, 3, 3, 3, 2, 3, 3, 3, 2, 3, 3, 2, 3, 3, 3, 2, 3, 3, 3, 2, 3, 3, 2, 3, 3, 3, 2, 3, 3, 2, 3, 3, 2, 3, 3, 3, 2, 3, 3, 2, 3, 2, 3, 3, 2, 3, 3, 2, 3, 3, 2, 3, 2, 3, 3, 2, 3, 3, 2, 3, 2, 3, 3, 2, 3, 2, 3, 3, 2, 3, 3, 2, 3, 2, 3, 3, 2, 3, 2, 3, 2, 3, 3, 2, 3, 2, 3, 3, 2, 3, 2, 3, 2, 3, 3, 2, 3, 2, 3, 2, 3, 2, 3, 3, 2, 3, 2, 2, 3, 2, 3, 2, 3, 2, 3, 2, 2, 3, 2, 3, 2, 3, 2, 2, 3, 2, 3, 2, 2, 3, 2, 3, 2, 3, 2, 2, 3, 2, 3, 2, 2, 3, 2, 2, 3, 2, 3, 2, 2, 3, 2, 3, 2, 2, 3, 2, 2, 3, 2, 3, 2, 2, 3, 2, 2, 3, 2, 2, 3, 2, 3, 2, 2, 3, 2, 2, 2, 3, 2, 2, 3, 2, 2, 3, 2, 2, 2, 3, 2, 2, 3, 2, 2, 2, 3, 2, 2, 2, 3, 2, 2, 3, 2, 2, 2, 3, 2, 2, 2, 2, 3, 2, 2, 2, 3, 2, 2, 2, 2, 3, 2, 2, 2, 2, 2, 3, 3, 2, 2, 2, 2, 2, 3, 2, 2, 2, 2, 3, 2, 2, 2, 3, 2, 2, 2, 2, 3, 2, 2, 2, 3, 2, 2, 3, 2, 2, 2, 3, 2, 2, 2, 3, 2, 2, 3, 2, 2, 2, 3, 2, 2, 3, 2, 2, 3, 2, 2, 2, 3, 2, 2, 3, 2, 3, 2, 2, 3, 2, 2, 3, 2, 2, 3, 2, 3, 2, 2, 3, 2, 2, 3, 2, 3, 2, 2, 3, 2, 3, 2, 2, 3, 2, 2, 3, 2, 3, 2, 2, 3, 2, 3, 2, 3, 2, 2, 3, 2, 3, 2, 2, 3, 2, 3, 2, 3, 2, 2, 3, 2, 3, 2, 3, 2, 3, 2, 2, 3, 2, 3, 3, 2, 3, 2, 3, 2, 3, 2, 3, 3, 2, 3, 2, 3, 2, 3, 3, 2, 3, 2, 3, 3, 2, 3, 2, 3, 2, 3, 3, 2, 3, 2, 3, 3, 2, 3, 3, 2, 3, 2, 3, 3, 2, 3, 2, 3, 3, 2, 3, 3, 2, 3, 2, 3, 3, 2, 3, 3, 2, 3, 3, 2, 3, 2, 3, 3, 2, 3, 3, 3, 2, 3, 3, 2, 3, 3, 2, 3, 3, 3, 2, 3, 3, 2, 3, 3, 3, 2, 3, 3, 3, 2, 3, 3, 2, 3, 3, 3, 2, 3, 3, 3, 3, 2, 3, 3, 3, 2, 3, 3, 3, 3, 2, 3, 3, 3, 3, 3, 2, 3, 4, 3, 3, 3, 3, 4, 3, 3, 3, 4, 3, 3, 4, 3, 3, 3, 4, 3, 3, 4, 3, 4, 3, 3, 4, 3, 3, 4, 3, 4, 3, 3, 4, 3, 4, 3, 4, 3, 3, 4, 3, 4, 4, 3, 4, 3, 4, 3, 4, 4, 3, 4, 3, 4, 4, 3, 4, 4, 3, 4, 3, 4, 4, 3, 4, 4, 4, 3, 4, 4, 3, 4, 4, 4, 3, 4, 4, 4, 4, 3, 4, 5, 4, 4, 4, 5, 4, 4, 5, 4, 5, 4, 4, 5, 4, 5, 5, 4, 5, 4, 5, 5, 4, 5, 5, 5, 4, 5, 6, 5, 5, 6, 5, 6, 6, 5, 6, 7, 6, 7, 8]\\

\end{document}